%% file: gradient_flow.tex
\documentclass[12pt]{article}

\usepackage[colorlinks]{hyperref}            
\usepackage{color}
\usepackage{graphicx,subfigure,amsmath,amssymb,amsfonts,bm,epsfig,epsf,url,dsfont}
\usepackage{fullpage}
\usepackage[small,bf]{caption}
\usepackage[top=1in,bottom=1in,left=1in,right=1in]{geometry}
\usepackage[sort&compress]{natbib} 
\usepackage{bbm}      
\usepackage{booktabs}
\usepackage{cases}
\usepackage{dsfont}
\usepackage{multirow}
\usepackage{graphicx}

\newtheorem{theorem}{Theorem}

\newtheorem{lemma}{Lemma}

\newenvironment{proof}{\medskip\noindent{\bf Proof.}}{\hfill$\Box$\vspace*{1mm}\medskip}

\include{commands}

\begin{document}
\title{How to trap a gradient flow}
\author{S\'ebastien Bubeck \\
Microsoft Research
\and Dan Mikulincer \thanks{This work was done while D. Mikulincer was an intern at Microsoft Research. Supported by an Azrieli foundation fellowship.} \\
Weizmann Institute}
\maketitle

\begin{abstract}
We consider the problem of finding an $\epsilon$-approximate stationary point of a smooth function on a compact domain of $\R^d$. In contrast with dimension-free approaches such as gradient descent, we focus here on the case where $d$ is finite, and potentially small. This viewpoint was explored in 1993 by Vavasis, who proposed an algorithm which, for {\em any fixed finite dimension $d$}, improves upon the $O(1/\epsilon^2)$ oracle complexity of gradient descent. For example for $d=2$, Vavasis' approach obtains the complexity $O(1/\epsilon)$. Moreover for $d=2$ he also proved a lower bound of $\Omega(1/\sqrt{\epsilon})$ for deterministic algorithms (we extend this result to randomized algorithms).  

Our main contribution is an algorithm, which we call {\em gradient flow trapping} (GFT), and the analysis of its oracle complexity. In dimension $d=2$, GFT closes the gap with Vavasis' lower bound (up to a logarithmic factor), as we show that it has complexity $O\left(\sqrt{\frac{\log(1/\epsilon)}{\epsilon}}\right)$. In dimension $d=3$, we show a complexity of $O\left(\frac{\log(1/\epsilon)}{\epsilon}\right)$, improving upon Vavasis' $O\left(1 / \epsilon^{1.2} \right)$. In higher dimensions, GFT has the remarkable property of being a {\em logarithmic parallel depth} strategy, in stark contrast with the polynomial depth of gradient descent or Vavasis' algorithm. In this higher dimensional regime, the total work of GFT improves quadratically upon the only other known polylogarithmic depth strategy for this problem, namely naive grid search. We augment this result with another algorithm, named \emph{cut and flow} (CF), which improves upon Vavasis' algorithm in any fixed dimension.

\end{abstract}
\section{Introduction}
Let $f : \R^d \rightarrow \R$ be a smooth function (i.e., the map $x \mapsto \nabla f(x)$ is $1$-Lipschitz, and $f$ is possibly non-convex). We aim to find an $\epsilon$-approximate stationary point, i.e., a point $x \in \R^d$ such that $\|\nabla f(x)\|_2 \leq \epsilon$. It is an elementary exercise to verify that for smooth and bounded functions, gradient descent finds such a point in $O(1/\epsilon^2)$ steps, see e.g., \cite{Nes04}. Moreover, it was recently shown in \cite{CDHS19} that this result is {\em optimal}, in the sense that any procedure with only black-box access to $f$ (e.g., to its value and gradient) must, {\em in the worst case}, make $\Omega(1/\epsilon^2)$ queries before finding an $\epsilon$-approximate stationary point. This situation is akin to the non-smooth convex case, where the same result (optimality of gradient descent at complexity $1/\epsilon^2$) holds true for finding an $\epsilon$-approximate optimal point (i.e., such that $f(x) - \min_{y \in \R^d} f(y) \leq \epsilon$), \cite{NY83, Nes04}. 
\newline

There is an important footnote to both of these results (convex and non-convex), namely that optimality only holds in {\em arbitrarily high dimension} (specifically the hard instance in both cases require $d = \Omega(1/\epsilon^2)$). It is well-known that in the convex case this large dimension requirement is actually necessary, for the cutting plane type strategies (e.g., center of gravity) can find $\epsilon$-approximate optimal points on compact domains in $O(d \log(1/\epsilon))$ queries. It is natural to ask: \textbf{Is there some analogue to cutting planes for non-convex optimization?}\footnote{We note that a different perspective on this question from the one developed in this paper was investigated in \citep{Hin18}, where the author asks whether one can adapt {\em actual cutting planes} to non-convex settings. In particular \cite{Hin18} shows that one can improve upon gradient descent and obtain a complexity $O(\mathrm{poly}(d)/\epsilon^{4/3})$ with a cutting plane method, under a higher order smoothness assumption (namely third order instead of first order here).} In dimension $1$ it is easy to see that one can indeed do a binary search to find an approximate stationary point of a smooth non-convex function on an interval. The first non-trivial case is thus dimension $2$, which is the focus of this paper (although we also obtain new results in high dimensions, and in particular our approach does achieve $O(\mathrm{poly}(d) \log(1/\epsilon))$ {\em parallel depth}, see below for details).
\newline

This problem, of finding an approximate stationary point of a smooth function on a compact domain of $\R^2$, was studied in 1993 by Stephen A. Vavasis in \citep{Vav93}. From an algorithmic perspective, his main observation is that in finite dimensional spaces one can speed up gradient descent by using a {\em warm start}. Specifically, observe that gradient descent only needs $O(\Delta/\epsilon^2)$ queries when starting from a $\Delta$-approximate optimal point. Leveraging smoothness (see e.g., Lemma \ref{lem:discretizationerror} below), observe that the best point on a $\sqrt{\Delta}$-net of the domain will be $\Delta$-approximate optimal. Thus starting gradient descent from the best point on $\sqrt{\Delta}$-net one obtains the complexity $O_d\left( \frac{\Delta}{\epsilon^2} + \frac{1}{\Delta^{d/2}} \right)$ in $\R^d$. Optimizing over $\Delta$, one obtains a $O_d\left(\left(\frac{1}{\epsilon}\right)^{\frac{2d}{d+2}}\right)$ complexity. In particular for $d=2$ this yields a $O(1/\epsilon)$ query strategy. In addition to this algorithmic advance, Vavasis also proved a lower bound of $\Omega(1/\sqrt{\epsilon})$ for deterministic algorithms. In this paper we close the gap up to a logarithmic term. Our main contribution is a new strategy loosely inspired by cutting planes, which we call {\em gradient flow trapping} (GFT), with complexity $O\left(\sqrt{\frac{\log(1/\epsilon)}{\epsilon}} \right)$. We also extend Vavasis lower bound to randomized algorithms, by connecting the problem with unpredictable walks in probability theory \citep{BPP98}.
\newline

Although we focus on $d=2$ for the description and analysis of GFT in this paper, one can in fact easily generalize to higher dimensions. Before stating our results there, we first make precise the notion of approximate stationary points, and we also introduce the {\em parallel query} model.

\subsection{Approximate stationary point} 
We focus on the constraint set $[0,1]^d$, although this is not necessary and we make this choice mainly for ease of exposition. Let us fix a differentiable function $f : [0,1]^d \rightarrow \R$ such that $\forall x,y \in [0,1]^d$, $\|\nabla f(x) - \nabla f(y)\|_2 \leq \|x-y\|_2$. Our goal is to find a point $x \in [0,1]^d$ such that for any $\epsilon' > \epsilon$, there exists a neighborhood $N \subset [0,1]^d$ of $x$ such that for any $y \in N$,
\[
f(x)  \leq f(y) + \epsilon' \cdot \|x-y\|_2 \,.
\]
We say that such an $x$ is an $\epsilon$-stationary point (its existence is guaranteed by the extreme value theorem).
In particular if $x \in (0,1)^d$ this means that $\|\nabla f(x)\|_2 \leq \epsilon$. More generally, for $x = (x^1, \hdots, x^d) \in [0,1]^d$ (possibly on the boundary), let us define the {\em projected gradient} at $x$, $g(x) = (g_1(x), \hdots, g_d(x))$ by:
\[
g_i(x) = \begin{cases}
\max \left(0, \frac{df}{dx^i}(x) \right)& \text{ if } x^i = 0 \,,\\
 \frac{df}{dx^i}(x) 			& \text{ if } x^i \in (0,1) \,,\\
\min\left(0, \frac{df}{dx^i}(x)\right )& \text{ if } x^i = 1 \,.
\end{cases}
\]
It is standard to show (see also \cite{Vav93}) that $x$ is an $\epsilon$-stationary point of $f$ if and only if $\|g(x)\|_2 \leq \epsilon$.

\subsection{Parallel query model}
In the classical black-box model, the algorithm can sequentially query an oracle at points $x \in [0,1]^d$ and obtain the value\footnote{Technically we consider here the zeroth order oracle model. It is clear that one can obtain a first order oracle model from it, at the expense of a multiplicative dimension blow-up in the complexity. In the context of this paper an extra factor $d$ is small, and thus we do not dwell on the distinction between zeroth order and first order.} of the function $f(x)$. An extension of this model, first considered in \citep{Nem94}, is as follows: instead of submitting queries one by one sequentially, the algorithm can submit any number of queries in parallel. One can then count the {\em depth}, defined as the number of rounds of interaction with the oracle, and the {\em total work}, defined as the total number of queries.

It seems that the parallel complexity of finding stationary points has not been studied before. As far as we know, the only low-depth algorithm (say depth polylogarithmic in $1/\epsilon$) is the naive grid search: simply query all the points on an $\epsilon$-net of $[0,1]^d$ (it is guaranteed that one point in such a net is an $\epsilon$-stationary point). This strategy has depth $1$, and total work $O(1/\epsilon^d)$. As we explain next, the high-dimensional version of GFT has depth $O(\mathrm{poly}(d) \log(1/\epsilon))$, and its total work improves at least quadratically upon grid search.

\subsection{Complexity bounds for GFT}
In this paper we give a complete proof of the following near-optimal result in dimension $2$:

\begin{theorem} \label{thm:main}
Let $d=2$. The gradient flow trapping algorithm (see Section \ref{sec:formal}) finds a $4 \epsilon$-stationary point with less than $10^5 \sqrt{\frac{\log(1/\epsilon)}{\epsilon}}$ queries to the value of $f$.
\end{theorem}

It turns out that there is nothing inherently two-dimensional about GFT. At a very high level, one can think of GFT as making hyperplane cuts, just like standard cutting planes methods in convex optimization. While in the convex case those hyperplane cuts are simply obtained by gradients, here we obtain them by querying a $\tilde{O}(\sqrt{\epsilon})$-net on a carefully selected small set of hyperplanes. Note also that the meaning of a ``cut" is much more delicate than for traditional cutting planes methods (here we use those cuts to ``trap" gradient flows). All of these ideas are more easily expressed in dimension $2$, but generalizing them to higher dimensions presents no new difficulties (besides heavier notation). In Section \ref{sec:highdim} we prove the following result:

\begin{theorem} \label{thm:highdim}
The high-dimensional version of GFT finds an $\epsilon$-stationary point in depth $O(d^2 \log(d/\epsilon))$ and in total work $d^{O(d)} \cdot \left(\frac{\log(1/\epsilon)}{\epsilon}\right)^{\frac{d-1}{2}} $.
\end{theorem}

In particular we see that the three-dimensional version of GFT has complexity $O\left(\frac{\log(1/\epsilon)}{\epsilon} \right)$. This improves upon the previous state of the art complexity $O(1/\epsilon^{1.2})$ \citep{Vav93}. However, on the contrary to the two-dimensional case, we believe that here GFT is suboptimal. As we discuss in Section \ref{sec:heuristic}, in dimension $3$ we conjecture the lower bound $\Omega(1/\epsilon^{0.6})$.

In dimensions $d \geq 4$, the total work given by Theorem \ref{thm:highdim} is worse than the total work $O\left(\left(\frac{1}{\epsilon}\right)^{\frac{2d}{d+2}}\right)$ of Vavasis' algorithm. On the other hand, the depth of Vavasis' algorithm is of the same order as its total work, in stark contrast with GFT which maintains a logarithmic depth even in higher dimensions. Among algorithms with polylogarithmic depth, the total work given in Theorem \ref{thm:highdim} is more than a quadratic improvement (in fixed dimension) over the previous state of the art (namely naive grid search).\\

We also propose a simplified version of GFT, which we call \emph{Cut and Flow} (CF), that always improve upon Vavasis' algorithm (in fact in dimension $d$ it attains the same rate as Vavasis in dimension $d-1$). In particular CF attains the same rate as GFT for $d=3$, and improves upon on it for any $d>3$. It is however a serial algorithm and does not enjoy the parallel properties of GFT.
\begin{theorem}\label{thm:secondary}
	Fix $d \in \mathbb{N}$. The cut and flow algorithm (see Section \ref{sec:first alg}) finds an $\epsilon$-stationary point with less than $5d^3\log\left(\frac{d}{\epsilon}\right)\left(\frac{1}{\epsilon}\right)^{\frac{2d-2}{d+1}}$ queries to the values of $f$ and $\nabla f$.
\end{theorem}
\subsection{Paper organization} 
The rest of the paper (besides Section \ref{sec:LB} and Section \ref{sec:open}) is dedicated to motivating, describing and analyzing our gradient flow trapping strategy in dimension $2$ (from now on we fix $d=2$, unless specified otherwise). In Section \ref{sec:basic} we make a basic ``local to global" observation about gradient flow which forms the basis of our ``trapping" strategy. Section \ref{sec:informal} is an informal section on how one could potentially use this local to global phenomenon to design an algorithm, and we outline some of the difficulties one has to overcome. As a warm-up, to demonstrate the use of our ideas, we introduce the "cut and flow" algorithm in Section \ref{sec:first alg} and prove Theorem \ref{thm:secondary}. In Section \ref{sec:formal} we formally describe our new strategy and analyze its complexity. In Section \ref{sec:LB} we extend Vavasis' $\Omega(1/\sqrt{\epsilon})$ lower bound to randomized algorithms. Finally we conclude the paper in Section \ref{sec:open} by introducing several open problems related to higher dimensions.

\section{A local to global phenomenon for gradient flow} \label{sec:basic}
We begin with some definitions. For an axis-aligned hyperrectangle $R =[a_1,b_1] \times\dots\times [a_d,b_d]$ in $\R^d$, we denote its volume and diameter by 
\[
\mathrm{diam}(R) : = \sqrt{\sum\limits_{i=1}^d(b_i - a_i)^2} \text{ and } \mathrm{vol}(R) : = \prod\limits_{i=1}^d(b_i - a_i)
\]
We further define the aspect ratio of $R$ as $\frac{\max_i\left(b_i-a_i\right)}{\min_i\left(b_i-a_i\right)}$.
The $2d$ faces of $R$ are the subsets of the form:
\[
[a_1,b_1]\times\dots \times \{a_i\} \times \dots \times [a_d,b_d] \text{ and } [a_1,b_1]\times\dots \times \{b_i\} \times \dots \times [a_d,b_d],
\]
for $i = 1,\dots, d$. The boundary of $R$, which we denote $\partial R$ is the union of all faces. \\

If $E \subset [0,1]^d$ is a $(d-1)$-dimensional hyperrectangle and $\delta > 0$, we say that $N \subset E$ is a $\delta$-net of $E$, if for any $x \in E$, there exists some $y \in N$ such that $\|x-y\|_2 \leq \delta$. We will always assume implicitly that if $N \subset E$ is a $\delta$-net, then the vertices of $E$ are elements of $N$.\\

We denote $f^*_{\delta}(E)$ for the largest value one can obtain by minimizing $f$ on a $\delta$-net of $E$. Formally,
\[
f^*_\delta(E) = \sup\limits_N\inf\limits_{x\in N}f(x),
\]
where the supremum is taken over all $\delta$-nets of $E$.
We say that a pair $(E,x)$ of segment/point in $[0,1]^d$ (where $E$ is {\em not} a subset of a face of $[0,1]^d$) satisfies the property $P_c$ for some $c\geq0$ if there exists $\delta>0$ such that
\[
f(x) < f^*_{\delta}(E) - \frac{\delta^2}{8}  + c \cdot \mathrm{dist}(x,E) \,,
\]
where 
\[
\mathrm{dist}(x,E) := \inf\limits_{y \in E} \|x-y\|_2.
\]
When $E$ is a subset of $\partial[0,1]^d$ we {\em always} say that $(E,x)$ satisfies $P_c$ (for any $c \geq 0$ and any $x \in [0,1]^d$).\\
For an axis-aligned hyperrectangle $R$ and $x \in R$, we say that $(R,x)$ satisfies $P_c$ if, for any of the $2d$ faces $E$ of $R$, one has that $(E,x)$ satisfies $P_c$. We refer to $x$ as the {\em pivot} for $R$. 

Our main observation is as follows:
\begin{lemma} \label{lem:main}
Let $R$ be a hyperrectangle such that $(R,x)$ satisfies $P_c$ for some $x \in R$ and $c \geq 0$. Then $R$ must contain a $c$-stationary point (in fact the gradient flow emanating from $x$ must visit a $c$-stationary point before exiting $R$).
\end{lemma}
This lemma will be our basic tool to develop cutting plane-like strategies for non-convex optimization. From ``local" information (values on a net of the boundary of $R$) one deduces a ``global" property (existence of approximate stationary point in $R$).

\begin{proof}
Let us assume by contradiction that $R$ does not contain a $c$-stationary point, and consider the unit-speed gradient flow $(x(t))_{t \geq 0}$ constrained to stay in $[0,1]^d$. That is, $x(t)$ is the piecewise differentiable function defined by $x(0)=x$ and $\frac{d}{dt} x(t) = - \frac{g(x(t))}{\|g(x(t))\|_2}$, where 
$g$ is the projected gradient defined in the previous section. Since there is no stationary point in $R$, it must be that the gradient flow exits $R$. Let us denote $T = \inf \{ t \geq 0 : x(t) \not\in R\}$, and $E$ a face of $R$ such that $x(T) \in E$.  Remark that $E$ cannot be part of a face of $[0,1]^d$. Furthermore, for any $0 \leq t \leq T$, one has 
\[
f(x(t)) - f(x(0)) = \int_{0}^t g(x(s)) \cdot \frac{d}{ds} x(s) ds \leq - c \cdot t \leq - c \cdot \|x(t) - x(0)\|_2 \,.
\]
where the first inequality uses that $R$ does not contain a $c$-stationary point. In particular, this implies $f(x(T)) - f(x) \leq - c \cdot \mathrm{dist}(x,E)$, so that,
\[
\min_{y \in E} f(y) \leq f(x) - c \cdot \mathrm{dist}(x, E) \,.
\]
Lemma \ref{lem:discretizationerror} below shows that for any $\delta >0$ one has $f^*_{\delta}(E) \leq \min_{y \in E} f(y) + \frac{\delta^2}{8}$, and thus together with the above display it shows that $(E,x)$ does {\em not} satisfy $P_c$, which is a contradiction.\\
\end{proof}

\begin{lemma} \label{lem:discretizationerror}
For any $(d-1)$-dimensional hyperrectangle $E \subset [0,1]^2$ and $\delta > 0$ one has:
\[
f^*_{\delta}(E) \leq \min_{y \in E} f(y) + \frac{\delta^2}{8}  \,.
\]
\end{lemma}

\begin{proof}
Let $x \in E$ be such that $f(x) = \min_{z \in E} f(z)$. If $x$ is a vertex of $E$, then we are done since we require the endpoints of $E$ to be in the $\delta$-nets. Otherwise $x$ is in the relative interior of $E$, and thus one has $\nabla f(x) \cdot (y - x) = 0$ for any $y \in E$. In particular by smoothness one has:
\begin{align*}
f(y) &= f(x) + \int_{0}^1 \nabla f(x + t (y-x)) \cdot (y-x) dt\\
 &\leq f(x) + \int_{0}^1 t\cdot \|y-x\|_2^2 dt = f(x) + \frac{1}{2} \|y-x\|_2^2 \,.
\end{align*}
Moreover for any $\delta$-net of $E$ there exists $y$ such that $\|y-x\|_2 \leq \frac{\delta}{2}$, and thus $f(y) \leq f(x) + \delta^2 / 8$, which concludes the proof.
\end{proof}

Our algorithmic approach to finding stationary points will be to somehow shrink the domain of consideration over time. At first it can be slightly unclear how the newly created boundaries interact with the definition of stationary points. To dispell any mystery, it might be useful to keep in mind the following observation, which states that if $(R,x)$ satisfies $P_c$, then $x$ cannot be on a boundary of $R$ which was not part of the original boundary of $[0,1]^d$.
\begin{lemma} \label{lem:minor}
Let $R$ be a rectangle such that $(R,x)$ satisfies $P_c$ for some $x \in R$ and $c \geq 0$. Then $x\notin \partial R \setminus \partial [0,1]^d$.
\end{lemma}

\begin{proof}
Let $E$ be a face of $R$ which is not a subset of $\partial [0,1]^d$. Then by definition of $P_c$, and by invoking Lemma \ref{lem:discretizationerror}, one has:
\[
f(x) < f_\delta^*(E) -\frac{\delta^2}{8} + c \cdot \mathrm{dist}(x,E) \leq \min_{y \in E} f(y) + c \cdot \mathrm{dist}(x,E) \,.
\]
In particular if $x \in E$ then $\mathrm{dist}(x,E) = 0$, and thus $f(x) < \min_{y \in E} f(y)$ which is a contradiction.
\end{proof}

\section{From Lemma \ref{lem:main} to an algorithm} \label{sec:informal}
Lemma \ref{lem:main} naturally leads to the following algorithmic idea (for sake of simplicity in this discussion we replace squares by circles): given some current candidate point $x$ in some well-conditioned domain (e.g., such that the domain contains and is contained in balls centered at $x$ and of comparable sizes), query a $\sqrt{\epsilon}$-net on the circle $C = \{ y : \|y-x\|_{2} = 1\}$, and denote $y$ for the best point found on this net. If one finds a significant enough improvement, say $f(y) < f(x) - \frac{3}{4} \epsilon$, then this is great news, as it means that one obtained a \textbf{per query} improvement of $\Theta(\epsilon^{-3/2})$ (to be compared with gradient descent which only yields an improvement of $\Theta(\epsilon^{-2})$). On the other hand if no such improvement is found, then the gradient flow from $x$ must visit an $\epsilon$-stationary point inside $C$.\footnote{In ``essence" $(C,x)$ satisfies $P_{\epsilon}$, this is only slightly informal since we defined $P_{\epsilon}$ for rectangles and $C$ is a circle. In particular we chose the improvement $\frac{3}{4} \epsilon$ instead of the larger $\frac{7}{8} \epsilon$ (which is enough to obtain $P_c$) to account for an extra term due to polygonal approximation of the circle. We encourage the reader to ignore this irrelevant technicality.} In other words one can now hope to restrict the domain of consideration to a region inside $C$, which is a constant fraction smaller than the original domain. Figure \ref{fig:figure0} illustrates the two possibilities. 
\newline

Optimistically this strategy would give a $\tilde{O}(B/\epsilon^{3/2})$ rate for $B$-bounded smooth functions (since at any given scale one could make at most $O(B/\epsilon^{3/2})$ improvement steps). In particular together with the warm start this would tentatively yield a $\tilde{O}(1/\epsilon^{3/4})$ rate, thus already improving the state-of-the-art $O(1/\epsilon)$ by Vavasis. 
\newline

\begin{figure}
	\centering
	\includegraphics[height = 4cm]{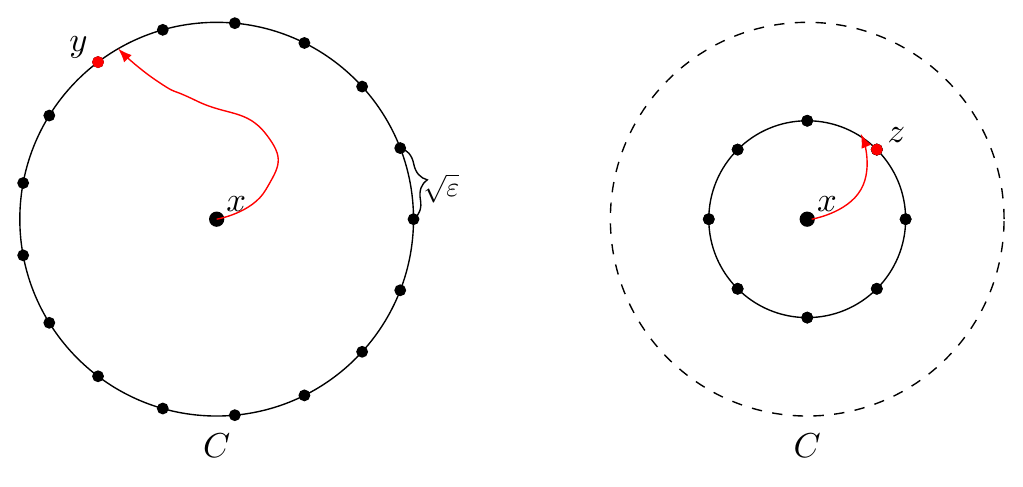}
	\caption{The red curve illustrates the gradient flow emanating from $x$. On the left, the flow does not visit an $\varepsilon$-stationary point and $y$ has a significantly smaller function value than $x$. Otherwise, as in the right, we shrink the domain.}
	\label{fig:figure0}
\end{figure}

There is however a difficulty in the induction part of the argument. Indeed, what we know after a shrinking step is that the current point $x$ satisfies $f(x) \leq f(y) + \epsilon$ for any $y \in C$. Now we would like to query a net on $\{y : \|y-x\|_{2} = 1/2\}$. Say that after such querying we find that we can't shrink, namely we found some point $z$ with $f(z) < f(x) - \frac{\epsilon}{2} + \frac{\delta^2}{8}$, and in particular $f(z) < f(y) + \frac{1}{2} \epsilon + \frac{\delta^2}{8}$ for any $y \in C$. Could the gradient flow from $z$ escape the original circle $C$ without visiting an $\epsilon$-stationary point? Unfortunately the answer is yes. Indeed (because of the discretization error $\delta^2/8$) one cannot rule out that there would be a point $y \in C$ with $f(y) < f(z) - \frac{\epsilon}{2}$, and since $C$ is only at distance $1/2$ from $z$, such a point could be attained from $z$ with a gradient flow without $\epsilon$-stationary points. Of course one could say that instead of satisfying $P_{\epsilon}$ we now only satisfy $P_{\epsilon + \delta^2 / 4}$, and try to control the increase of the approximation guarantee, but such an approach would not improve upon the $1/\epsilon^2$ of gradient descent (simply because we could additively worsen the approximation guarantee too many times).
\newline

The core part of the above argument will remain in our full algorithm (querying a $\sqrt{\epsilon}$-net to shrink the domain). However it is made more difficult by the discretization error as we just saw. We also note that this discretization issue does not appear in discrete spaces, which is one reason why discrete spaces are much easier than continuous spaces for local optimization problems.
\newline

Technically we observe that the whole issue of discretization comes from the fact that when we update the center, we move closer to the boundary, which we ``pay" in the term $\mathrm{dist}(x,E)$ in $P_c$, and we cannot ``afford" it because of the discretization error term that we suffer when we update. Thus this issue would disappear if in our induction hypothesis we had $P_0$ for the boundary. Our strategy will work in two steps: first we give a querying strategy for a domain with $P_0$ that ensures that one can \textbf{always} shrink with $P_{\epsilon}$ guaranteed for the boundary, and secondly we give a method to essentially turn a $P_{\epsilon}$ boundary into $P_0$.
\section{Cut and flow} \label{sec:first alg}
We now fix $d \in \mathbb{N}$ and consider $[0,1]^d$. We say that a pair $(H,x)$ is a \emph{domain} if $H \subset [0,1]^d$ is an axis-aligned hyperrectangle and $x \in H$. In this section, we further require that if $H = [a_1, b_1]\times\dots\times[a_d, b_d]$, then for every $1 \leq i,j \leq d$, $\frac{b_i - a_i}{b_j - a_j} \in \{\frac{1}{2},1,2\}$. In other words, all edges of $H$ either have the same length or differ by a factor of $2$. The Cut and Flow (CF) algorithm is performed with two alternating steps, \emph{bisection} and \emph{descent} (See Figure \ref{fig:figurebisection} for an illustration of the two steps, when $d = 2$).
\begin{enumerate}
	\item At the \emph{bisection} step, we have a domain $(H,x)$ satisfying $P_0$. Let $k\in [d]$ be any coordinate such that $b_k - a_k$ is maximal and set the midpoint, $m_k = \frac{a_k + b_k}{2}$. We now bisect $H$ into two equal parts,
	\begin{align*}
	H_1 &= [a_1,b_1]\times\dots \times[a_k, m_k]\times\dots\times[a_d,b_d],\\
	H_2 &= [a_1,b_1]\times\dots \times[m_k, b_k]\times\dots\times[a_d,b_d],
	\end{align*}
	so that $H_1 \cup H_2 = H$ and $E = H_1 \cap H_2$ is a $(d-1)$-dimensional hyperrectangle. Set $N \subset E$ to be a $\delta$-net and,
	$$x_N = \arg\min\limits_{y\in N}f(y).$$
	Here $\delta$ is some small parameter to determined later.
	To choose a new pivot $\bar{x}$ for the domain we compare $f(x_N)$ and $f(x)$. If $f(x) \leq f(x_N)$, set $\bar{x} = x$ otherwise $\bar{x} = x_N$. 
	We end the step with the two pairs $(H_1, \bar{x})$, $(H_2, \bar{x})$.
	\item  The \emph{descent} step takes the two pairs produced by the \emph{bisection} step and returns a new domain $(\tilde{H}, \tilde{x})$ satisfying $P_0$ such that $\tilde{H} \in \{H_1, H_2\}.$ This is done be performing gradient descent iterations:
	\begin{equation} \label{eq:GD}
	\bar{x}_i = \bar{x}_{i-1} -\nabla f(\bar{x}_{i-1}),
	\end{equation}
	where $\bar{x}_0 = \bar{x}$.
	Set $T = \frac{\delta^2}{\epsilon^2}$, and $\tilde{x} = \bar{x}_T$. Then, $\tilde{H} = H_1$ if $\tilde{x} \in H_1$ and $\tilde{H} = H_2$ otherwise.
\end{enumerate}
The CF algorithm starts with the domain $(H_0,x_0)$ where $H_0 = [0,1]^d$ and $x_0$ is arbitrary. Given $(H_t, x_t)$ the algorithm runs a \emph{bisection} step, followed by a \emph{descent} step and sets $(H_{t+1},x_{t+1}) = (\tilde{H},\tilde{x})$, as described above. The algorithm stops when the diameter of $H_t$ is smaller than $\varepsilon$.\\
\begin{figure}[h!]
	\centering
	\includegraphics[width=9.55cm]{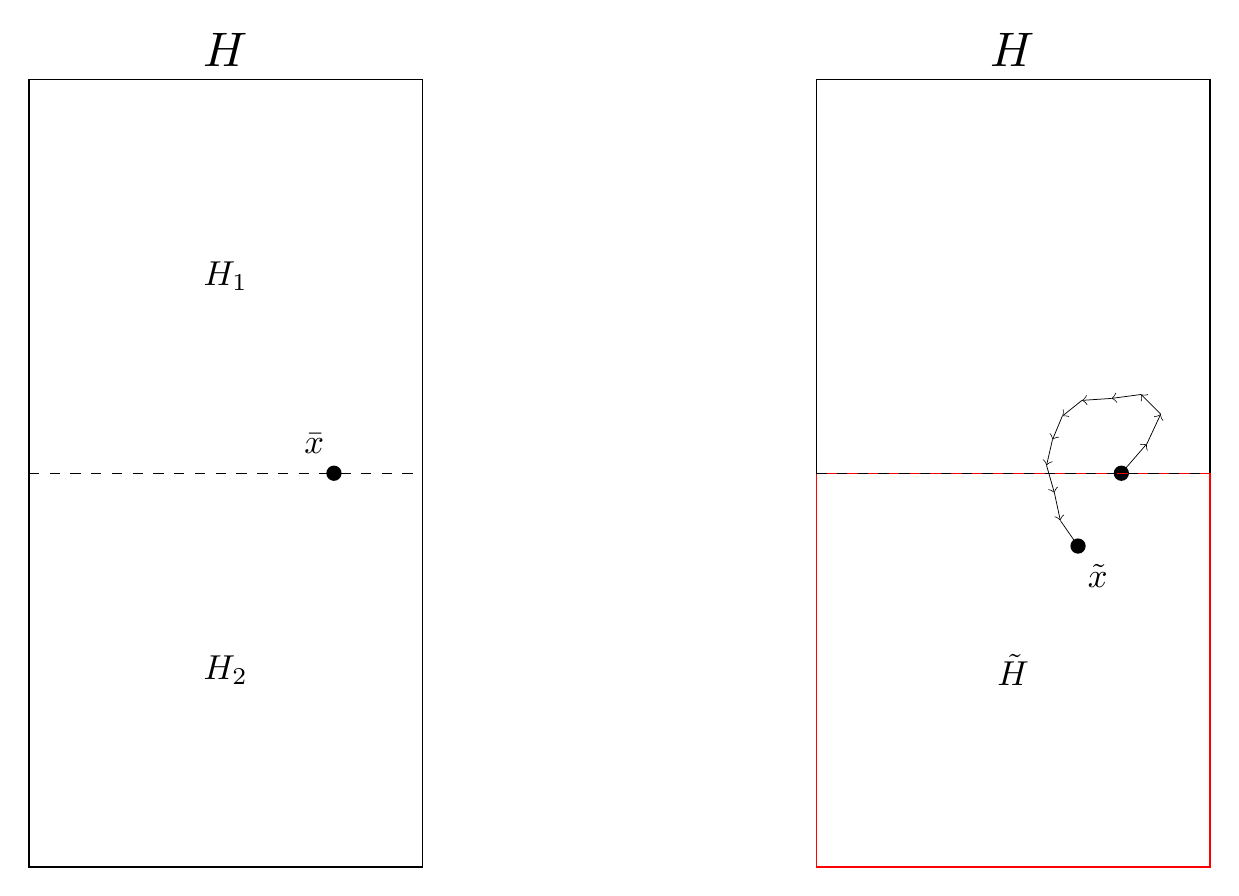}
	\caption{The left image shows the bisection of $H$ into two equal parts $H_1$ and $H_2$. The right image shows the trajectory of gradient descent, starting from $\bar{x}$ and terminating at $\tilde{x}$, inside $\tilde{H}$.}
	\label{fig:figurebisection}
\end{figure}
\newpage
Let us first prove that at the end of the \emph{descent} step, the obtained domain satisfies $P_0$.
\begin{lemma} \label{lem:p0bisection}
	Suppose that $(H,x)$ satisfies $P_0$, then either the \emph{descent} step finds an $\epsilon$-stationary point or $(\tilde{H},\tilde{x})$ satisfies $P_0$ as well.
\end{lemma}
\begin{proof}
	Let us first estimate the value of $f(\tilde{x})$. Observe that, by smoothness of $f$, if we consider the gradient descent iterates \eqref{eq:GD}, we have
	$$f(\bar{x}_{i-1})-f(\bar{x}_{i}) \geq \|\nabla f(\bar{x}_{i-1})\|^2_2 - \frac{1}{2}\|\nabla f(\bar{x}_{i-1})\|^2_2 =\frac{1}{2}\|\nabla f(\bar{x}_{i-1})\|^2_2\geq \frac{\varepsilon^2}{2},$$
	where the last inequality holds as long as $\bar{x}_{i-1}$ is not an $\varepsilon$-stationary point (see also Section 3.2 in \cite{bubeck2015convex}). It follows that,
	$$f(\tilde{x}) = f(\bar{x}_T) \leq f(\bar{x}) - \frac{T}{2}\epsilon^2 \leq f(x) - \frac{T}{2}\epsilon^2.$$
	Now, let $E' \subset \partial\tilde{H}$ be a face such that $E' \neq H_1 \cap H_2$. Then, $E' \subset \partial H$ and by assumption, $(E', x)$ satisfied $P_0$. Since $f(\tilde{x}) \leq f(x)$, it is clear that $(E', \tilde{x})$ satisfies $P_0$ as well.
	
	We are left with showing that, if $E = H_1 \cap H_2$, then $(E,\tilde{x})$ satisfies $P_0$.
	Indeed, from the construction, and using $T = \frac{\delta^2}{\epsilon^2}$, we have,
	$$f(\tilde{x}) \leq f(\bar{x}) - \frac{T}{2}\epsilon^2 \leq f_\delta^*(E) - \frac{T}{2}\epsilon^2 \leq f_\delta^*(E) - \frac{\delta^2}{8}.$$
\end{proof}

Let us now prove Theorem \ref{thm:secondary}.

\begin{proof}[of Theorem \ref{thm:secondary}]
	Observe that $\mathrm{diam}(H_0) = \sqrt{d}$ and that after performing $d$ consecutive \emph{bisection} steps, necessarily, every face of $H_t$ was bisected into two equal parts. 
	Hence, $\mathrm{diam}(H_{t+d}) \leq \frac{1}{2}\mathrm{diam}(H_{t})$, and,
	$$\mathrm{diam}(H_t) \leq \left(\frac{1}{2}\right)^{\lfloor\frac{t}{d}\rfloor}\sqrt{d}.$$
	Choose $T = \lceil d\log_2\left(\frac{\sqrt{d}}{\epsilon}\right)\rceil$, so that $\mathrm{diam}(H_T) \leq \epsilon$. We claim that $x_T$ is an $\epsilon$-stationary point. Indeed, by iterating Lemma \ref{lem:p0bisection} we know that the pair $(H_T, x_T)$ satisfies $P_0$. By Lemma \ref{lem:main}, there exists $x_*\in H_T$ which is a stationary point and $\|x_* - x_T\|_2 \leq \epsilon$.\\
	
	All that remains is to calculate the number of queries made by the algorithm. At the \emph{bisection} step we query a $\delta$-net $N$, over a $(d-1)$-dimensional hyperrectangle, contained in the unit cube. Elementary computations show that we can take,
	$$|N| \leq \frac{\left(2d\right)^{d-1}}{\delta^{d-1}}.$$
	Combined with the number of queries made by the \emph{descent step}, we see that the total number of queries made by the algorithm is,
	$$\left\lceil d\log_2\left(\frac{\sqrt{d}}{\epsilon}\right)\right\rceil\left(\frac{\left(2d\right)^{d-1}}{\delta^{d-1}} + \frac{\delta^2}{\epsilon^2} \right).$$
	We now optimize and choose $\delta = \epsilon^{\frac{2}{d+1}} 2d$. Substituting into the above equations shows that the number of queries is smaller than
	$$5d^3\log_2\left(\frac{d}{\epsilon}\right)\epsilon^{-\frac{2d-2}{d+1}}.$$
\end{proof}
\section{Gradient flow trapping} \label{sec:formal}
In this section we focus on the case $d=2$. We say that a pair $(R,x)$ is a {\em domain} if $R$ is an axis-aligned rectangle with aspect ratio bounded by $3$, and $x \in R$ (note that the definition of a domain is slightly different than in the previous section). The gradient flow trapping (GFT) algorithm is decomposed into two subroutines:
\begin{enumerate}
\item The first algorithm, which we call the {\em parallel trap}, takes as input a domain $(R,x)$ satisfying $P_0$. It returns a domain $(\tilde{R}, \tilde{x})$ satisfying $P_{\epsilon}$ and such that $\mathrm{vol}(\tilde{R}) \leq 0.95 \ \mathrm{vol}(R)$. The cost of this step is at most $2 \sqrt{\frac{\mathrm{diam}(R)}{\epsilon}}$ queries.
\item The second algorithm, which we call {\em edge fixing}, takes as input a domain $(R,x)$ satisfying $P_{\epsilon'}$ (for some $\epsilon' \in [\epsilon, 2 \epsilon]$) and such that for $k \in \{0,1,2,3\}$ edges $E$ of $R$ one also has $P_0$ for $(E,x)$. It returns a domain $(\tilde{R}, \tilde{x})$ such that either (i) it satisfies $P_{\epsilon'}$ and for $k+1$ edges it also satisfies $P_0$, or (ii) it satisfies $P_{\left(1+ \frac{1}{500 \log(1/\epsilon)} \right) \epsilon'}$ and furthermore $\mathrm{vol}(\tilde{R}) \leq 0.95 \ \mathrm{vol}(R)$. The cost of this step is at most $90 \sqrt{\frac{\mathrm{diam}(R) \log(1/\epsilon)}{\epsilon}}$ queries.
\end{enumerate}
Equipped with these subroutines, GFT proceeds as follows. Initialize $(R_0, x_0) = ([0,1]^2, (0.5,0.5))$, $\epsilon_0 = \epsilon$, and $k_0 = 4$. For $t \geq 0$: 
\begin{itemize}
\item if $k_t =4$, call {\em parallel trap} on $(R_t, x_t)$, and update $k_{t+1}=0$, $(R_{t+1},x_{t+1})=(\tilde{R_t}, \tilde{x_t})$, and $\epsilon_{t+1} = \epsilon$.
\item Otherwise call {\em edge fixing}, and update $(R_{t+1},x_{t+1})=(\tilde{R_t}, \tilde{x_t})$. If $R_{t+1} = R_t$ then set $k_{t+1}=k_t+1$ and $\epsilon_{t+1} = \epsilon_t$, and otherwise set $k_{t+1} = 0$ and $\epsilon_{t+1} = \left(1+ \frac{1}{500 \log(1/\epsilon)} \right) \epsilon_t$.
\end{itemize}
We terminate once the diameter of $R_t$ is smaller than $2 \epsilon$.

Next we give the complexity analysis of GFT assuming the claimed properties of the subroutines {\em parallel trap} and {\em edge fixing} in 1. and 2. above. We then proceed to describe in details the subroutines, and prove that they satisfy the claimed properties.

\subsection{Complexity analysis of GFT} \label{sec:complexity}
The following three lemmas give a proof of Theorem \ref{thm:main}.

\begin{lemma} \label{lem:proof1}
GFT stops after at most $200 \log(1/\epsilon)$ steps.
\end{lemma}

\begin{proof}
First note that at least one out of five steps of GFT reduces the volume of the domain by $0.95$ (since one can do at most four steps in a row of edge fixing without volume decrease). Thus on average the volume decrease per step is at least $0.99$, i.e., $\mathrm{vol}(R_T) \leq 0.99^T$. In particular since $R_T$ has aspect ratio smaller than $3$, it is easy to verify $\mathrm{diam}(R_T) \leq 2\sqrt{\mathrm{vol}(R_T)} \leq 2 \times 0.99^{T/2}$. Thus for any $T \geq \log_{100/99}(1/\epsilon^{2})$, one must have $\mathrm{diam}(R_T) \leq 2 \epsilon$. Thus we see that GFT performs at most $\log_{100/99}(1/\epsilon^{2}) \leq 200 \log(1/\epsilon)$ steps.
\end{proof}

\begin{lemma}
When GFT stops, its pivot is a $4 \epsilon$-stationary point.
\end{lemma}

\begin{proof}
First note that $\epsilon_T \leq \left(1+\frac{1}{500 \log(1/\epsilon)}\right)^T \epsilon$, thus after $T \leq 200 \log(1/\epsilon)$ steps we know that $(R_T,x_T)$ satisfies at least $P_{2 \epsilon}$. In particular by Lemma \ref{lem:main}, $R_T$ must contain a $2 \epsilon$-stationary point, and since the diameter is less than $2 \epsilon$, it must be (by smoothness) that $x_T$ is a $4 \epsilon$-stationary point.
\end{proof}

\begin{lemma}
GFT makes at most $10^5 \sqrt{\frac{\log(1/\epsilon)}{\epsilon}}$ queries before it stops.
\end{lemma}

\begin{proof}
As we saw in the proof of Lemma \ref{lem:proof1}, one has $\mathrm{diam}(R_t) \leq 2 \times 0.99^{t/2}$. Furthermore the $t^{th}$ step requires at most $90 \sqrt{\frac{\mathrm{diam}(R_t) \log(1/\epsilon)}{\epsilon}}$ queries. Thus the total number of queries is bounded by:
\[
90 \sqrt{\frac{2 \log(1/\epsilon)}{\epsilon}} \sum_{t=0}^{\infty} 0.99^{t/4} \leq 10^5 \sqrt{\frac{\log(1/\epsilon)}{\epsilon}} \,.
\]
\end{proof}

\subsection{A parallel trap} \label{sec:paralleltrap}
Let $(R,x)$ be a domain. We define two segments $E$ and $F$ in $R$ as follows. Assume that $R$ is a translation of $[0,s] \times [0,r]$. For sake of notation assume that in fact $R=[0,s] \times [0,r]$ with $s \in [r, 3r]$ and $x^1 \geq r/2$, where $x = (x^1,x^2)$ (in practice one always ensures this situation with a simple change of variables). Now we define $E = \{r/6\} \times [0,r]$ and $F=\{r/3\} \times [0,r]$ (See Figure \ref{fig:figure1}).
\begin{figure}[h!]
	\centering
	\includegraphics[width=10cm]{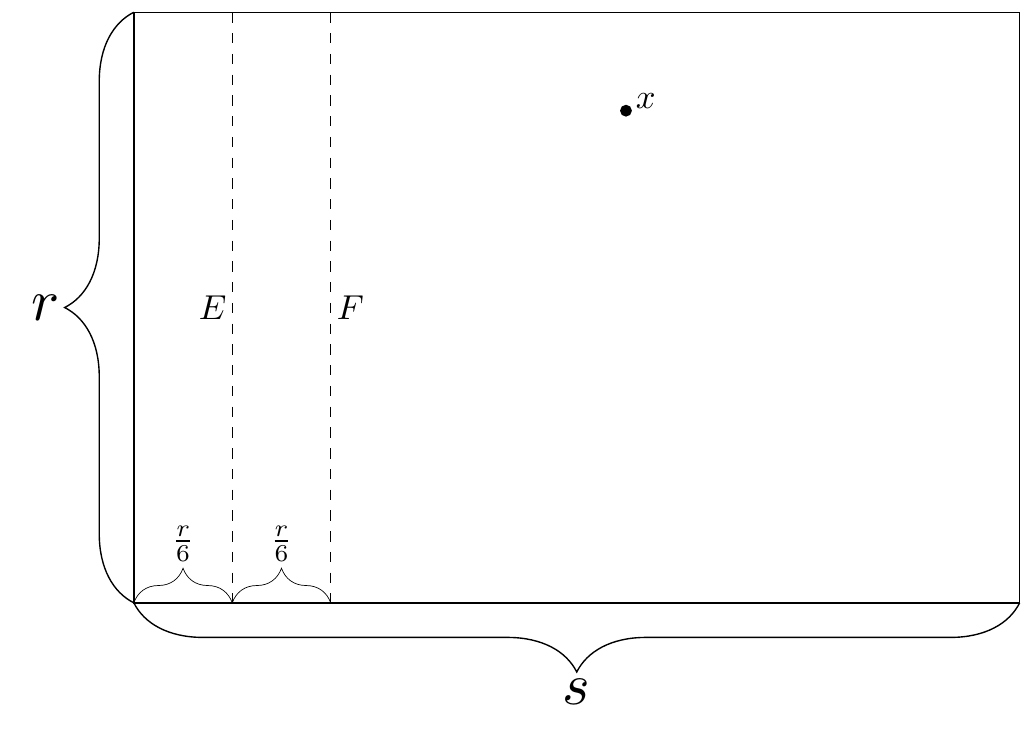}
	\caption{The \emph{parallel trap}}
	\label{fig:figure1}
\end{figure}

The parallel trap algorithm queries a $\sqrt{r \epsilon}$-net on both $E$ and $F$ (which cost at most $2 \frac{r}{\sqrt{r \epsilon}} = 2 \sqrt{\frac{r}{\epsilon}}$). Denote $\bar{x}$ to be the best point (in terms of $f$ value) found on the union of those nets. That is, denoting $N \subset F \cup E$ for the queried $\sqrt{r\epsilon}$-net, then
\[
\bar{x} = \arg\min\limits_{y\in N} f(y).
\]
 One has the following possibilities (see Figure \ref{fig:figure2} for an illustration):
\begin{itemize}
\item If $f(x) \leq f(\bar{x})$ then we set $\tilde{x}= x$ and $\tilde{R} = [r/3,s] \times [0,r]$.
\item Otherwise we set $\tilde{x} = \bar{x}$. If $\bar{x} \in F$ we set $\tilde{R} = [r/6,s] \times [0,r]$, and if $\bar{x} \in E$ we set $\tilde{R} = [0,r/3] \times [0,r]$.
\end{itemize}
The above construction is justified by the following lemma (a trivial consequence of the definitions), and it proves in particular the properties of {\em parallel trap} described in 1. at the beginning of Section \ref{sec:formal}.
\begin{lemma}
The rectangle $\tilde{R}$ has aspect ratio smaller than $3$, and it satisfies $\mathrm{vol}(\tilde{R}) \leq 0.95 \ \mathrm{vol}(R)$. Furthermore if $(R,x)$ satisfies $P_0$, then $(\tilde{R}, \tilde{x})$ satisfies $P_{\epsilon}$.
\end{lemma}

\begin{proof}
The first sentence is trivial to verify. For the second sentence, first note that for any edge $E$ of $R$ one has $P_0$ for $(E,\tilde{x})$ since by assumption one has $P_0$ for $(E,x)$ and furthermore $f(\tilde{x}) \leq f(x)$. Next observe that $\tilde{R}$ has at most one new edge $\tilde{E}$ with respect to $R$, and this edge is at distance at least $r/6$ from $\tilde{x}$, thus in particular one has $\epsilon \cdot \mathrm{dist}(\tilde{x}, \tilde{E}) - \delta^2 / 8 > 0$ for $\delta=\sqrt{r \epsilon}$. Furthermore by definition $f(\tilde{x}) \leq f^*_{\delta}(\tilde{E})$, and thus $f(\tilde{x}) < f^*_{\delta}(\tilde{E}) - \frac{\delta^2}{8} + \epsilon \cdot \mathrm{dist}(\tilde{x}, \tilde{E})$, or in other words $(\tilde{E},\tilde{x})$ satisfies $P_{\epsilon}$. 
\end{proof}

\begin{figure}[h]
	\includegraphics[width=\columnwidth]{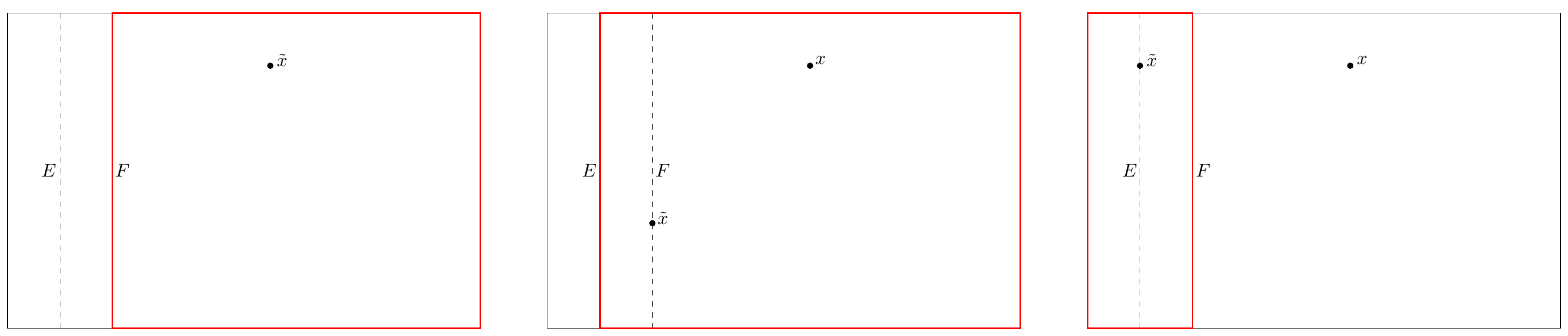}
	\caption{The three possible cases for $(\tilde{R},\tilde{x})$. $\tilde{R}$ is marked in red.}
	\label{fig:figure2}
\end{figure}
\subsection{Edge fixing} \label{sec:edgefixing}
Let $(R,x)$ be a domain satisfying $P_{\epsilon'}$ for some $\epsilon' \in [\epsilon, 2 \epsilon]$, and with some edges possibly also satisfying $P_0$. Denote $\mathcal{E}$ for the closest edge to $x$ that does not satisfy $P_0$, and let $r=\mathrm{dist}(x,\mathcal{E})$. We will consider three\footnote{We need three candidates to ensure that the domain will shrink.} candidate smaller rectangles, $R_1$, $R_2$ and $R_3$, as well as three candidate pivots (in addition to $x$) $x_1 \in \partial R_1$, $x_2 \in \partial R_2$ and $x_3 \in \partial R_3$. The rectangles are defined by $R_i = R \cap \{y : \|x_{i-1}-y\|_{\infty} \leq \frac{r}{3} \}$, where we set $x_0 = x$. The possible output $(\tilde{R}, \tilde{x})$ of {\em edge fixing} will be either $(R_i, x_{i-1})$ for some $i\in\{1,2,3\}$, or $(R, x_3)$ (see Figure \ref{fig:figure3} for a demonstration of how to construct the rectangles).

To guarantee the properties described in 2. at the beginning of Section \ref{sec:formal} we will prove the following: if the output is $(R_i, x_{i-1})$ for some $i$ then all edges will satisfy $P_{\left( 1 + \frac{1}{500 \log(1/\epsilon)} \right) \epsilon'}$ (Lemma \ref{lem:Ri} below) and the domain has shrunk (Lemma \ref{lem:shrunkdomain} below), and if the output is $(R,x_3)$ then one more edge satisfies $P_0$ compared to $(R,x)$ while all edges still satisfy at least $P_{\epsilon'}$ (Lemma \ref{lem:notRi} below).

\begin{lemma} \label{lem:shrunkdomain}
For any $i \in \{1,2,3\}$ one has $\mathrm{vol}(R_i) \leq \frac{2}{3} \mathrm{vol}(R)$. Furthermore if the aspect ratio of $R$ is smaller than $3$, then so is the aspect ratio of $R_i$.
\end{lemma}

\begin{proof}
Let us denote $\ell_1(R)$ for the length of $R$ in the axis of $\mathcal{E}$ (the edge whose distance to $x$ defines $r$), and $\ell_2(R)$ for the length in the orthogonal direction (and similarly define $\ell_1(R_i)$ and $\ell_2(R_i)$).

Since $R_i \subset R$ one has $\ell_1(R_i) \leq \ell_1(R)$. Furthermore $\ell_2(R_i) \leq \frac{2}{3} r$ and $\ell_2(R) \geq r$, so that $\ell_2(R_i) \leq \frac{2}{3} \ell_2(R)$. This implies that $\mathrm{vol}(R_i) \leq \frac{2}{3} \mathrm{vol}(R)$.

For the second statement observe that $\ell_1(R) \geq \frac{\ell_2(R)}{3} \geq \frac{r}{3}$ (the first inequality is by assumption on the aspect ratio of $R$, the second inequality is by definition of $r$). Given this estimate, the construction of $R_i$ implies that $\frac{1}{3}r\leq \ell_2(R_i),\ell_1(R_i) \leq \frac{2}{3}r$, which concludes the fact that $R_i$ has aspect ratio smaller than $2$. 
\end{proof}

\paragraph{Queries and choice of output.}
The edge fixing algorithm queries a $\sqrt{\frac{\epsilon' r}{500\log(1/\epsilon)}}$-net on $\partial R_i$ for all $i \in \{1,2,3\}$ (thus a total of $4 \sqrt{\frac{500 r \log(1/\epsilon)}{\epsilon'}} \leq 90 \sqrt{\frac{r \log(1/\epsilon)}{\epsilon}}$ queries), and we define $x_i$ to be the best point found on each respective net.

If for all $i \in \{1,2,3\}$ one has  
\begin{equation} \label{eq:noimprov}
f(x_i) \leq f(x_{i-1}) - \frac{\epsilon' r}{3} \,,
\end{equation}
then we set $(\tilde{R},\tilde{x})=(R,x_3)$. Otherwise denote $i^* \in \{1,2,3\}$ for the smallest number which violates \eqref{eq:noimprov}, and set $(\tilde{R},\tilde{x})=(R_{i^*},x_{i^*-1})$.
\begin{figure}
	\centering
	\includegraphics[height = 6cm]{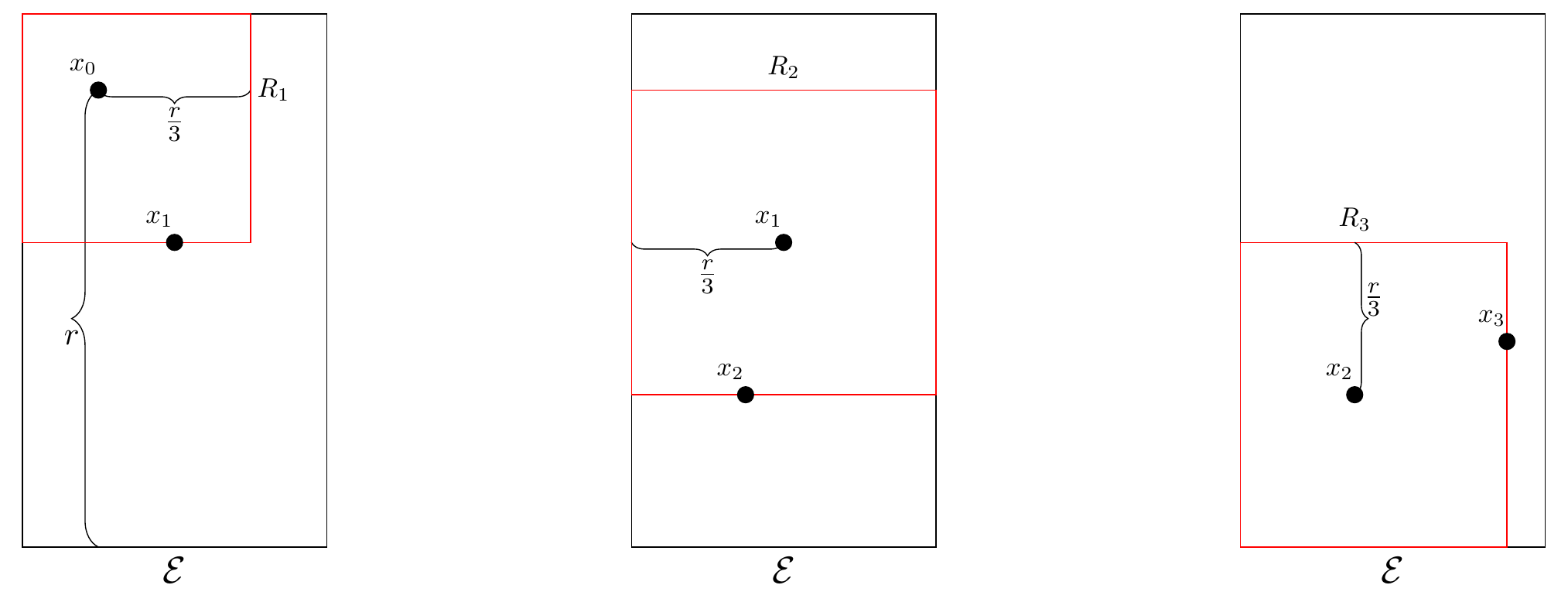}
	\caption{\emph{Edge fixing}: the rectangles $R_1, R_2$ and $R_3$ are marked in red, from left to right.}
	\label{fig:figure3}
\end{figure}
\begin{lemma} \label{lem:notRi}
If $(\tilde{R},\tilde{x})=(R,x_3)$ then $(\mathcal{E},x_3)$ satisfies $P_0$. Furthermore for any edge $E$ of $R$, if $(E,x)$ satisfies $P_0$ (respectively $P_{\epsilon'}$) then so does $(E,x_3)$.
\end{lemma}

\begin{proof}
Since $(\tilde{R},\tilde{x})=(R,x_3)$ it means that $f(x_3) \leq  f(x_0) - \epsilon' r$. In particular since $(\mathcal{E},x_0)$ satisfies $P_{\epsilon'}$ one has $f(x_0) < f^*_{\delta}(\mathcal{E}) - \frac{\delta^2}{8} + \epsilon' r$, and thus now one has $f(x_3) < f^*_{\delta}(\mathcal{E}) - \frac{\delta^2}{8}$ which means that $(\mathcal{E},x_3)$ satisfies $P_0$.

Let us now turn to some other edge $E$ of $R$. Certainly if $(E,x_0)$ satisfies $P_0$ then so does $(E,x_3)$ since $f(x_3) \leq f(x_0)$. But, in fact, even $P_{\epsilon'}$ is preserved since by the triangle inequality (and $\|x_3 - x_0\|_2 \leq r$) one has
 \[
f(x_3) - \epsilon' \cdot \mathrm{dist}(x_3,E) \leq f(x_3) + \epsilon' r - \epsilon' \cdot \mathrm{dist}(x_0,E) \leq f(x_0) - \epsilon' \cdot \mathrm{dist}(x_0,E).
\]
\end{proof}

\begin{lemma} \label{lem:Ri}
If $(\tilde{R},\tilde{x})=(R_{i},x_{i-1})$ for some $i \in \{1,2,3\}$, then $(\tilde{R},\tilde{x})$ satisfy $P_{\left( 1 + \frac{1}{500 \log(1/\epsilon)} \right) \epsilon'}$.
\end{lemma}

\begin{proof}
By construction, if $(\tilde{R},\tilde{x})=(R_{i},x_{i-1})$, then for any edge $E$ of $R_i$ one has $f(x_{i-1}) < f^*_{\delta}(E) + \frac{\epsilon' r}{3}$. Furthermore one has $\frac{\epsilon' r}{3} = - \frac{\epsilon' r}{8 \times 500 \log(1/\epsilon)} + \left(1+ \frac{3}{8 \times 500 \log(1/\epsilon)} \right) \frac{\epsilon' r}{3}$, and thus one has $P_{\left(1+ \frac{3}{8 \times 500 \log(1/\epsilon)} \right) \epsilon'}$ for $(E, x_{i-1})$ whenever $\mathrm{dist}(x_{i-1}, E) = \frac{r}{3}$. Indeed, since $\delta = \sqrt{\frac{\epsilon' r}{500\log(1/\epsilon)}}$,
\[
f(x_{i-1}) < f^*_{\delta}(E)-\frac{\delta^2}{8} + \left(1+ \frac{3}{8 \times 500 \log(1/\epsilon)} \right) \epsilon'\cdot\mathrm{dist}\left(x_{i-1},E\right).
\]
If $\mathrm{dist}(x_{i-1}, E) < \frac{r}{3}$ then by the triangle inequality, $\mathrm{dist}(x_0,E) < r$, and moreover $E$ is also an edge with respect to $R$. Thus from the definition of $r$, $(E,x_0)$ satisfies $P_0$. Also by our choice of $x_{i-1}$, we know that $f(x_{i-1}) \leq f(x_0)$. Hence $(E,x_{i-1})$ satisfies $P_0$ as well. 
\end{proof}

\subsection{Generalization to higher dimensions} \label{sec:highdim}
As explained in the introduction, there is no reason to restrict GFT to $[0,1]^2$ and, in fact, the algorithm may be readily adapted to higher-dimensional spaces, such as $[0,1]^d$, for some $d > 2$. We now detail the necessary changes and derive the complexity announced in Theorem \ref{thm:highdim}. 

First, if $F$ is an affine hyperplane, and $x \in [0,1]^d$, we define $P_c$ for $(F,x)$ in the obvious way (i.e., same definition except that we consider a $\delta$-net of $F$). Similarly for $(R,x)$, when $R$ is an axis-aligned hyperrectangle. 

Gradient flow trapping in higher dimensions replaces every line by a hyperplane, and every rectangle by a hyperrectangle. In particular at each step GFT maintains a domain $(R,x)$, where $R$ is a hyperrectangle with aspect ratio bounded by 3, and $x \in R$. The two subroutines are adapted as follows:
\begin{enumerate}
	\item {\em Parallel trap} works exactly in the same way, except that the two lines $E$ and $F$ are replaced by two corresponding affine hyperplanes. In particular the query cost of this step is now $O\left(\left(\frac{\mathrm{diam}(R)}{\epsilon}\right)^\frac{d-1}{2}\right)$, and the volume shrinks by at least $0.95$.
	\item In {\em edge fixing}, we now have three hyperrectangles $R_i$, and we need to query nets on their $2d$ faces. Thus the total cost of this step is  $O \left( d\left(\frac{\mathrm{diam(R)}\log(1/\epsilon)}{\epsilon}\right)^\frac{d-1}{2} \right)$.
	Moreover, suppose that domain does not shrink at the end of this step and the output is a domain $(R,\tilde{x})$ for some other $\tilde{x} \in R$. In this case we know that $R$ has some face $F$, such that $(F,x)$ did not satisfy $P_0$, but $(F,\tilde{x})$ does satisfy $P_0$. It follows that we can run {\em edge fixing}, at most $2d$ times before the domain shrinks.
\end{enumerate}
We can now analyze the complexity of the high-dimensional version of GFT:

\begin{proof}[Of Theorem \ref{thm:highdim}]
First observe that, if $R$ is a hyperrectangle in $[0,1]^d$ with aspect ratio bounded by $3$, then we have the following inequality,
$$\mathrm{diam}(R) \leq 3\sqrt{d}\cdot\mathrm{vol}(R)^{\frac{1}{d}}.$$
By repeating the same calculations done in Lemma \ref{lem:proof1} and the observation about \emph{parallel trap} and \emph{edge fixing} made above, we see that the domain shrinks at least once in every $2d + 1$ steps, so that at step $T$,
$$\mathrm{vol}(R_T) \leq 0.95^{\frac{T}{2d+1}},$$
and 
$$\mathrm{diam}(R_T) \leq 3\sqrt{d}\cdot 0.95^{\frac{T}{(2d+1)d}}.$$
Since the algorithm stops when $\mathrm{diam}(R_T) \leq 2\epsilon$, we get
$$ T = O\left(d^2\log\left(\frac{d}{\epsilon}\right)\right).$$
The total work done by the algorithm is evident now by considering the number of queries at each step.
\end{proof}

\section{Lower bound for randomized algorithms} \label{sec:LB}
In this section, we show that any randomized algorithm must make at least $\tilde{\Omega}\left(\frac{1}{\sqrt{\epsilon}}\right)$ queries in order to find an $\epsilon$-stationary point. This extends the lower bound in \citep{Vav93}, which applied only to deterministic algorithms. In particular, it shows that, up to logarithmic factors, adding randomness cannot improve the algorithm described in the previous section.
\newline

For an algorithm $\mathcal{A}$, a function $f: [0,1]^2\to \R$  and $\epsilon >0$ we denote by $\mathcal{Q}\left(\mathcal{A},f,\epsilon\right)$ the number of queries made by $\mathcal{A}$, in order to find an $\epsilon$-stationary point of $f$. Our goal is to bound from below
$$\mathcal{Q}_{\mathrm{rand}}(\epsilon):=\inf\limits_{\mathcal{A} \text{ random} }\sup\limits_f\E_\mathcal{A}\left[\mathcal{Q}\left(\mathcal{A},f,\epsilon\right)\right],$$
where the infimum is taken over all random algorithms and the supremum is taken over all smooth functions, $f$. The expectation is with respect to the randomness of $\mathcal{A}$. By Yao's minimax principle we have the equality
$$\mathcal{Q}_{\mathrm{rand}}(\epsilon) = \sup\limits_{\mathcal{D}}\inf\limits_{\mathcal{A} \text{ determinstic}}\E_{f\sim\mathcal{ D}}\left[\mathcal{Q}\left(\mathcal{A},f,\epsilon\right)\right].$$
Here, $\mathcal{A}$ is a deterministic algorithm and $\mathcal{D}$ is a distribution over smooth functions.
The rest of this section is devoted to proving the following theorem:
\begin{theorem} \label{thm: lower bound}
	Let $h:\N \to \R$ be a decreasing function such that
	$$\sum\limits_{k=1}^\infty \frac{h(k)}{k}< \infty,$$
	and set 
	\begin{equation} \label{eq: def sum}
	S_h(n):= \sum\limits_{k=1}^n \frac{1}{k\cdot h(k)}.
	\end{equation}
	Then,
	$$\mathcal{Q}_\mathrm{rand}(\epsilon)=\Omega\left(\frac{1}{\sqrt{\epsilon}\cdot S_h\left(\left\lceil\frac{1}{\sqrt{\epsilon}}\right\rceil\right)}\right).$$
\end{theorem}
Remark that one may take $h(k):= \frac{1}{\log(k)^2+1}$ in the theorem. In this case $S_h(k) = O(\log^3(k))$, and
$\mathcal{Q}_\mathrm{rand}(\epsilon)=\Omega\left(\frac{1}{\log^3(1/\epsilon) \sqrt{\epsilon} }\right)$, which is the announced lower bound.\\

One of the main tools utilized in our proof is the construction introduced in \citep{Vav93}. We now present the relevant details.
\subsection{A reduction to monotone path functions}
Let $G_n = (V_n,E_n)$ stand for the $n+1 \times n+1$ grid graph. That is, 
$$V_n = \{0,\hdots, n\} \times \{0,\hdots, n\} \text{ and } E_n =\{(v,u)\in V_n\times V_n : \|v-u\|_1 = 1\}.$$ We say that a sequence of vertices, $(v_0,...,v_n)$ is a \emph{monotone path} in $G_n$ if $v_0 = (0,0)$ and for every $0<i\leq n$, $v_i - v_{i-1}$ either equals $(0,1)$ or $(1,0)$. In other words, the path starts at the origin and continues each step by either going right or up. 
If $(v_0,...,v_n)$ is a monotone path, we associate to it a \emph{monotone path function} $P:V_n \to \R$ by
$$P(v) = \begin{cases}
-\|v\|_1& \text{if } v\in \{v_0,...,v_n\}\\
\hfill \|v\|_1& \text{otherwise}
\end{cases}.$$ 
By a slight abuse of notation, we will sometimes refer to the path function and the path itself as the same entity. If $i = 0,...,n$ we write $P_i$ for $P^{-1}(-i)$ and $P[i]$ for the prefix $(P_0,P_1,...,P_{i})$. If $v \in V_n$ is such that $P(v) > 0$, we say that $v$ does not lie on the path.\\
We denote the set of all monotone path functions on $G_n$ by
$\mathrm{F}_n.$
It is clear that if $P \in \mathrm{F}_n$ then $P_n$ is the only local minimum of $P$ and hence the global minimum. \\

Informally, the main construction in \citep{Vav93} shows that for every $P \in \mathrm{F}_n$ there is a corresponding smooth function $\hat{P}:[0,1]^2\to \R$, which 'traces' the path in $P$ and preserves its structure. In particular, finding an $\epsilon$-stationary point of $\hat{P}$ is not easier than finding the minimum of $P$.\\

To formally state the result we fix $\epsilon >0$ and assume for simplicity that $\frac{1}{\sqrt{\epsilon}}$ is an integer. We henceforth denote $n(\epsilon) := \frac{1}{\sqrt{\epsilon}}$ and identify $V_{n(\epsilon)}$ with $[0,1]^2$ in the following way: if $(i,j) = v\in V_{n(\epsilon)}$ we write $\mathrm{square}(v)$ for the square:
$$\mathrm{square}(v) = \left[\frac{i}{n(\epsilon) + 1}, \frac{i+1}{n(\epsilon) + 1}\right] \times \left[\frac{j}{n(\epsilon) + 1}, \frac{j+1}{n(\epsilon) + 1}\right].$$
If $\varphi: [0,1]^2 \to \R$, then $\mathrm{supp}(\varphi)$ denotes the closure of the set $\{x \in[0,1]^2 : \varphi(x) \neq 0\}$.
\begin{lemma}[Section 3, \citep{Vav93}] \label{lem: vavsis}
	Let $P \in \mathrm{F}_{n(\epsilon)}$. Then there exists a function $\hat{P}:[0,1]^2\to\R$ with the following properties:
	\begin{enumerate}
		\item $\hat{P}$ is smooth.
		\item $\hat{P} = f_P + \ell$, where $\ell$ is a linear function, which does not depend on $P$, and 
		$$\mathrm{supp}(f_P)\subset \bigcup\limits_{i=0}^n\mathrm{square}\left(P_i\right).$$
		\item If $x \in [0,1]^2$ is an $\epsilon$-stationary point of $\hat{P}$ then $x \in \mathrm{square}\left(P_n\right)$.
		\item if $P' \in \mathrm{F}_{n(\epsilon)}$ is another function and for some $i = 0,...,n$, $(P'_{i-1},P'_i,P'_{i+1}) = (P_{i-1},P_i,P_{i+1})$. Then
		$$\hat{P}'\vert_{\mathrm{square}\left(P_i\right)} = \hat{P}\vert_{\mathrm{square}\left(P_i\right)}$$
	\end{enumerate}
\end{lemma}
We now make precise of the fact that finding the minimum of $P$ is as hard as finding an $\epsilon$-stationary point of $\hat{P}$. For this we define $\mathcal{G}(\mathcal{A},P)$, the number of queries made by algorithm $\mathcal{A}$, in order to find the minimal value of the function $P$. 
\begin{lemma} \label{lem: reduction}
	For any algorithm $\mathcal{A}$, which finds an $\epsilon$-stationary point of smooth functions on $[0,1]^2$, there exists an algorithm $\tilde{\mathcal{A}}$ such that
	$$\mathcal{Q}(\mathcal{A}, \hat{P},\epsilon) \geq \frac{1}{5} \mathcal{G}(\tilde{\mathcal{A}},P),$$
	for any $P \in \mathrm{F}_{n(\epsilon)}.$
\end{lemma}
\begin{proof}
	Given an algorithm $\mathcal{A}$ we explain how to construct $\tilde{\mathcal{A}}$. Fix $P \in  \mathrm{F}_{n(\epsilon)}$. If $\mathcal{A}$ queries a point $x \in \mathrm{square}\left(v\right)\subset [0,1]^2$. Then $\tilde{\mathcal{A}}$ queries $v$ and all of its neighbors. When $\mathcal{A}$ terminates it has found an $\epsilon$-stationary point. By Lemma \ref{lem: vavsis}, this point must lie in $\mathrm{square}\left(P_n\right)$. By querying $P_n$ and its neighbors, $\tilde{\mathcal{A}}$ will determine that $P_n$ is a local minimum and hence the minimum of $P$.\\
	
	Since each vertex has at most $4$ neighbors, it will now suffice to show that $\tilde{\mathcal{A}}$ can remain consistent with $\mathcal{A}$. We thus need to show that after querying the neighbors of $v$, $\tilde{\mathcal{A}}$ may deduce the value of $\hat{P}(x)$.\\
	
	As we are only interested in the number of queries made by $\tilde{\mathcal{A}}$, it is fine to assume that $\tilde{\mathcal{A}}$ has access to the construction used in Lemma \ref{lem: vavsis}. Now, suppose that $P(v) > 0$ and $v$ does not lie on the path. In this case, by Lemma \ref{lem: vavsis}, $\hat{P}(x) = \ell(x)$, which does not depend on $P$ itself and $\ell(x)$ is known. Otherwise $v = P_i$ for some $i =0,...,n$. So, after querying the neighbors of $v$, $\tilde{\mathcal{A}}$ also knows $P_{i-1}$ and $P_{i+1}$. The lemma then tells us that $\hat{P}\vert_{\mathrm{square}\left(v\right)}$ is uniquely determined and, in particular, the value of $\hat{P}(x)$ is known.
\end{proof}
\subsection{A lower bound for monotone path functions}
Denote $\mathcal{D}_{\mathrm{p}}(n)$ to be the set of all distributions supported on $\mathrm{F}_n$. By Lemma \ref{lem: reduction},
$$\mathcal{Q}_{\mathrm{rand}}(\epsilon) \geq \sup\limits_{\mathcal{D}\in\mathcal{D}_{\mathrm{p}}\left(n(\epsilon)\right) }\inf\limits_{\mathcal{A} \text{ determinstic}}\E_{P\sim\mathcal{ D}}\left[\mathcal{Q}\left(\mathcal{A},\hat{P},\epsilon\right)\right] \geq \frac{1}{5} \sup\limits_{\mathcal{D}\in\mathcal{D}_{\mathrm{p}}\left(n(\epsilon)\right) }\inf\limits_{\mathcal{A} \text{ determinstic}}\E_{P\sim\mathcal{ D}}\left[\mathcal{G}\left(\tilde{\mathcal{A}},P\right)\right].$$
In \citep{sun2009quantum}, the authors present a family of random paths $(X_\delta)_{\delta>0} \subset \mathcal{D}_{\mathrm{p}}(n)$. Using these random paths it is shown that for every $\delta > 0$,
$$\mathcal{G}_{\mathrm{rand}}(n):=\sup\limits_{\mathcal{D}\in\mathcal{D}_{\mathrm{p}}\left(n\right) }\inf\limits_{\mathcal{A} \text{ determinstic}}\E_{P\sim\mathcal{ D}}\left[\mathcal{G}\left(\mathcal{A},P\right)\right] = \Omega(n^{1-\delta}).$$
This immediately implies,
$$\mathcal{Q}_{\mathrm{rand}}(\epsilon) = \Omega\left(\left(\frac{1}{\sqrt{\epsilon}}\right)^{1-\delta}\right).$$
Their proof uses results from combinatorial number theory in order to construct a random path which, roughly speaking, has unpredictable increments. This distribution is then used in conjunction with a method developed by Aaronson (\citep{aaronson2006lower}) in order to produce a lower bound.\\

We now present a simplified proof of the result, which also slightly improves the bound. We simply observe that known results concerning unpredictable random walks, can be combined with Aaronson's method.  Theorem \ref{thm: lower bound} will then be a consequence of the following theorem:
\begin{theorem} \label{thm: discrete lower bound}
	Let the notations of Theorem \ref{thm: lower bound} prevail. Then
	$$\mathcal{G}_\mathrm{rand}(n) = \Omega\left(\frac{n}{S_h(n)}\right).$$
\end{theorem}
The theorem of Aaronson, reformulated using our notations (see also \citep[Lemma 2]{sun2009quantum}), is given below.
\begin{theorem}[Theorem 5, \citep{aaronson2006lower}] \label{thm: aaronson}
	Let $w :\mathrm{F}_n \times \mathrm{F}_n \to \R^+$ be a weight function with the following properties:
	\begin{itemize}
		\item $w(P,P') = w(P',P)$.
		\item $w(P,P') = 0$, whenever $P_n = P'_n$.
	\end{itemize}
	Define 
	$$T(w,P) := \sum\limits_{Q \in \mathrm{F}_n} w(P,Q),$$ and for $v \in V_n$
	$$T(w,P,v):= \
	\sum\limits_{Q \in \mathrm{F}_n: Q(v) \neq P(v)}w(P,Q).$$
	Then
	$$\mathcal{G}_{\mathrm{rand}}(n) = \Omega\left( \min\limits_{\substack{P,P',v\\P(v)\neq P'(v), w(P,P') > 0}}\max\left(\frac{T(w,P)}{T(w,P,v)},\frac{T(w,P')}{T(w,P',v)}\right)\right).$$
\end{theorem}
For $P \in \mathrm{F}_n$, one should think about $w$ as inducing a probability measure according to $w(P,\cdot)$.
If $Q$ is sampled according to this measure, then the quantity $\frac{T(w,P,v)}{T(w,P)}$ is the probability that $P(v) \neq Q(v)$. That is, either $v \in P$ or $v \in Q$, but not both. The theorem then says that if this probability is small, for at least one path in each pair $(P, P')$ such that $P_n \neq P'_n$, then any randomized algorithm must make as many queries as the reciprocal of the probability. \\

We now formalize this notion; For a random path $X \in \mathcal{D}_\mathrm{p}(n)$, define the following weight function:
$$w_X(P,P') = \begin{cases}
0& \text { if } P_n = P'_n\\
\P(X = P)\cdot\sum\limits_{i=0}^{n-1}\P(X = P'| X[i] = P[i])& \text{ otherwise }
\end{cases}.$$
Here $w_X(P,P')$ is proportional to the probability that $X = P'$, conditional on agreeing with $P$ on the first $i$ steps, where $i$ is uniformly chosen between $0$ and $n-1$.
Note that, for any $i$,
\begin{align*}
\P\left(X = P\right)\cdot&\P\left(X = P'| X[i] = P[i]\right) \\
&=\P(X[i] = P[i])\cdot\P(X = P|X[i] = P[i])\cdot\P(X = P'| X[i] = P[i])\\
&=\P\left(X = P'\right)\cdot\P\left(X = P| X[i] = P'[i]\right).
\end{align*}
Hence, $w_X(P, P') = w_X(P', P)$.
We will use the following theorem from \cite{haggstrom1998nearest}, which generalizes the main result of \citep{BPP98}.
\begin{theorem}[Theorem 1.4, \citep{haggstrom1998nearest}]
	Let $h$ be as in Theorem \ref{thm: lower bound},
	Then there exists a random path $X^h \in \mathcal{D}_\mathrm{p}(n)$ and a constant $c_h>0$, such that for all $m\geq k$, and for every $(v_0,v_1...,v_{m-k})$, sequence of vertices,
	\begin{equation} \label{eq: unpredictablity}
	\sup\limits_{\|u\|_1 = m}\P\left(X^h_m = u|X^h_0 = v_0,...,X^h_{m-k} = v_{m-k}\right) \leq \frac{c_h}{kf(k)}.
	\end{equation}
\end{theorem}
For $X^h$ as in the theorem abbreviate $w_h := w_{X^h}$ and recall $S_h(n) : =  \sum\limits_{k=1}^{n}\frac{1}{k\cdot h(k)}$. We now prove the main quantitative estimates which apply to $w_h$.
\begin{lemma}
	For any $P \in \mathrm{F}_n$, 
	$$\sum\limits_{Q \in \mathrm{F}_n}w_h(P,Q) \geq \P(X^h = P)\cdot\left(n - c_hS_h(n)\right).$$
\end{lemma}
\begin{proof}
	We write
	\begin{align*}
	\sum\limits_{Q \in \mathrm{F}_n}w_h(P,Q) &= \P(X^h = P)\sum\limits_{Q: Q_n \neq P_n}\sum\limits_{i=0}^{n-1}\P\left(X^h = Q| X^h[i] = P[i]\right)\\
	&=\P(X^h = P)\sum\limits_{i=0}^{n-1}\left(1 - \sum\limits_{Q: Q_n = P_n}\P\left(X^h = Q| X^h[i] = P[i]\right)\right)\\
	&=\P(X^h = P)\left(n - \sum\limits_{i=0}^{n-1}\sum\limits_{Q: Q_n = P_n}\P\left(X^h = Q| X^h[i] = P[i]\right)\right).
	\end{align*}
	Using \eqref{eq: unpredictablity}, we get
	$$\sum\limits_{Q: Q_n = P_n}\P\left(X^h = Q| X^h[i] = P[i]\right) \leq \P(X^h_n = P_n|X^h[i] = P[i])\leq \frac{c_h}{(n-i)\cdot h(n-i)},$$
	and 
	$$\sum\limits_{i=0}^{n-1}\sum\limits_{Q: Q_n = P_n}\P\left(X^h = Q| X^h[i] = P[i]\right) \leq \sum\limits_{k=1}^{n}\frac{c_h}{k\cdot h(k)} = c_hS_h(n).$$
\end{proof}
\begin{lemma}
	Let $P \in \mathrm{F}_n$ and $v \in V_n$ such that $\|v\|_1 = \ell$ and $P_{\ell} \neq v$. Then,
	$$\sum\limits_{\substack{Q\in \mathrm{F}_n\\ Q_\ell =v}}w_h(P,Q) \leq 2\P(X^h = P)c_hS_h(n).$$
\end{lemma}
\begin{proof}
	\begin{align*}
	\sum\limits_{\substack{Q\in \mathrm{F}_n\\ Q_\ell =v}}w_h(P,Q) &=\P(X^h = P)\sum\limits_{i=0}^{\ell - 1}\sum\limits_{\substack{Q: Q_n \neq P_n\\Q_\ell = v}}\P\left(X^h = Q| X^h[i] = P[i]\right)\\
	&\leq\P(X^h = P)\sum\limits_{i=0}^{\ell - 1}\sum\limits_{\substack{Q: Q_\ell = v
	}}\P\left(X^h = Q| X^h[i] = P[i]\right).
	\end{align*}
	Observe that if $Q_\ell = v$, then $Q_{\ell+1}$ must equal $v + (0,1)$ or $v + (1,0)$.
	In particular, for $i < \ell$, \eqref{eq: unpredictablity} shows
	\begin{align*}
	\sum\limits_{\substack{Q: Q_{\ell} = v
	}}\P&\left(X^h =Q| X^h[i] = P[i]\right) \\
	&\leq \P\left(X^h_{\ell+1} = v + (0,1) \text{ or }X^h_{\ell+1} = v + (1,0)|X^h[i] = P[i]\right)\\
	&\leq \P\left(X^h_{\ell+1} = v + (0,1)|X^h[i] = P[i]\right) + \P\left(X^h_{\ell+1} = v + (1,0)|X^h[i] = P[i]\right)\\
	&\leq \frac{2c_h}{(\ell+1 - i)\cdot h(\ell+1 - i)}.
	\end{align*}
	So,
	\begin{align*}
	\sum\limits_{i=0}^{\ell - 1}\sum\limits_{\substack{Q: Q_\ell = v
	}}\P\left(X^h =Q| X^h[i] = P[i]\right) &\leq  \sum\limits_{i=0}^{\ell - 1}\frac{2c_h}{(\ell+1 - i)\cdot h(\ell+1 - i)}\\
	\leq 2c_hS_h(n).
	\end{align*}
\end{proof}

We are now in a position to prove Theorem \ref{thm: discrete lower bound}.

\begin{proof}[of Theorem \ref{thm: discrete lower bound}]
	Let $P \in \mathrm{F}_n$ and let $v \in V_n$, with 
	$$\|v\|_1 = \ell \text{ and } P_\ell \neq v.$$
	 Note that $P(v) = \ell$. So, if $Q \in \mathrm{F}_n$ is such that $Q(v) \neq P(v)$, then necessarily $Q_{\ell} = v$. We now set $P' \in \mathrm{F}_n$, with $P(v) \neq P'(v)$. In this case, the previous two lemmas show
	 \begin{align*}
	\max&\left(\frac{T(w_h,P)}{T(w_h,P,v)},\frac{T(w_h,P')}{T(w_h,P',v)}\right) \\
	&\geq \frac{T(w_h,P)}{T(w_h,P,v)} 
	= \frac{\sum\limits_{Q\in \mathrm{F}_n}w_h(P,Q)}{\sum\limits_{\substack{Q\in \mathrm{F}_n\\Q(v) \neq  P(v)}}w_h(P,Q)} = \frac{\sum\limits_{Q\in \mathrm{F}_n}w_h(P,Q)}{\sum\limits_{\substack{Q\in \mathrm{F}_n\\Q_{\ell}=v}}w_h(P,Q)}\geq \frac{n - c_hS_h(n)}{2c_hS_h(n)}.
	\end{align*}
	Since we are trying to establish a lower bound, we might as well assume that $S_h(n) = o(n)$. So, for $n$ large enough
	$$\frac{n - c_hS_h(n)}{2c_hS_h(n)} \geq\frac{n}{4c_hS_h(n)}.$$
	Plugging this estimate into Theorem \ref{thm: aaronson} yields the desired result
\end{proof}

\subsection{Heuristic extension to higher dimensions} \label{sec:heuristic}
In this section we propose a heuristic approach to extend the lower bound to higher dimensions. In the $2$ dimensional case, the proof method of Section \ref{sec:LB} consisted of two steps: first reduces the problem to the discrete setting of monotone paths in $[n]^2$, and then analyze the query complexity of finding the minimal point for such path functions. Thus, to extend the result we should consider path functions on the $d$-dimensional grid, as well as a way to build smooth functions on $[0,1]^d$ from those paths.

The lower bound for finding minimal points of path functions in high-dimensional grids was obtained in \citep{Zha06}, where it was shown that, in the worst case, any randomized algorithm must make $\Omega\left(n^{\frac{d}{2}}\right)$ queries in order to find the end point of a path defined over $[n]^{d}$. Thus, if we can find a discretization scheme, analogous to Lemma \ref{lem: vavsis}, in higher dimensions, we could obtain a lower bound for finding $\epsilon$-stationary points. What are the constraints on such a discretization?

First note that necessarily the construction of \citep{Zha06} must be based on paths of lengths $\Omega\left(n^{\frac{d}{2}}\right)$, for otherwise one could simply trace the path to find its endpoint. In particular, since each cube has edge length $\frac{1}{n}$, an analogous construction to Lemma \ref{lem: vavsis} will reach value smaller than $- \epsilon \cdot n^{\frac{d}{2}-1}$ at the stationary point (i.e., the endpoint of the path). On the other hand, in at least one of the neighboring cubes (which are at distance less than $1/n$ from the stationary point), the background linear function should prevail, meaning that the function should reach a positive value. Since around the stationary point the function is quadratic, we get the constraint:
\[
- \epsilon \cdot n^{\frac{d}{2}-1} + \left(\frac{1}{n}\right)^2 > 0 \Leftrightarrow n < \left( \frac{1}{\epsilon} \right)^{\frac{2}{d+2}} \,.
\]
In particular the lower bound $\Omega\left(n^{\frac{d}{2}}\right)$ now suggests that for finding stationary point one has the complexity lower bound $\left( \frac{1}{\epsilon} \right)^{\frac{d}{d+2}}$.

\section{Discussion} \label{sec:open}
In this paper we introduced a near-optimal algorithm for finding $\epsilon$-stationary points in dimension $2$. Finding a near-optimal algorithm in dimensions $d \geq 3$ remains open. Specific challenges include:
\begin{enumerate}
\item Finding a strategy in dimension $3$ which improves upon GFT's $\tilde{O}(1/\epsilon)$ complexity.
\item The heuristic extension of the lower bound in Section \ref{sec:heuristic} suggests $\Omega\left(\frac{1}{\epsilon^{\frac{d}{d+2}}}\right)$ as a complexity lower bound for any dimension $d$ (note in particular that the exponent tends to $1$ as $d$ tends to infinity). On the other hand, \cite{CDHS19} proved that for $d = \Omega(1/\epsilon^2)$, one has the complexity lower bound $\Omega(1/\epsilon^2)$. How do we reconcile these two results?  Specifically we raise the following question: Is there an algorithm with complexity $C_d / \epsilon$ for some constant $C_d$ which depends only on $d$? (Note that $C_d$ as small as $O(\sqrt{d})$ would remain consistent with \cite{CDHS19}.) Alternatively we might ask whether the \cite{CDHS19} lower bound holds for much smaller dimensions, e.g. when $d = \Theta(\log(1/\epsilon))$, are we in the $1/\epsilon$ regime as suggested by the heuristic, or are we already in the high-dimensional $1/\epsilon^2$ of \cite{CDHS19}?
\item Especially intriguing is the limit of low-depth algorithms, say as defined by having depth smaller than $\mathrm{poly}(d \log(1/\epsilon))$. Currently this class of algorithms suffers from the curse of dimensionality, as GFT's total work degrades significantly when the dimension increases (recall from Theorem \ref{thm:highdim} that it is $\tilde{O}\left(\frac{1}{\epsilon^{\frac{d-1}{2}}}\right)$). Is this necessary? A much weaker question is to simply show a separation between low-depth and high-depth algorithms. Namely can one show a lower bound $\Omega(1/\epsilon^c)$ with $c>2$ for low-depth algorithms? We note that lower bounds on depth have been investigated in the convex setting, see \citep{Nem94}, \citep{BJLLS19}.
\item A technically challenging problem is to adapt the construction in [Section 3, \cite{Vav93}] to non-monotone paths in higher dimensions. In particular, to formalize the heuristic argument from Section \ref{sec:heuristic}, such construction should presumably avoid creating saddle points.
\end{enumerate}
Many more questions remain open on how to exploit the low-dimensional geometry of smooth gradient fields, and the above four questions are only a subset of the fundamental questions that we would like to answer. Other interesting questions include closing the logarithmic gap in dimension $2$, or understanding better the role of randomness for this problem (note that GFT is deterministic, but other type of strategies include randomness, such as Hinder's non-convex cutting plane \citep{Hin18}).

\subsubsection*{Acknowledgment}
We thank Ronen Eldan, Yin Tat Lee, Yuanzhi Li and Mark Selke for many helpful discussions on this problem.

\bibliographystyle{plainnat}
\bibliography{bib}

\end{document}

%% file: commands.tex
\renewcommand{\phi}{\varphi}

\renewcommand{\P}{\mathbb{P}}
\newcommand{\E}{\mathbb{E}}
\newcommand{\N}{\mathbb{N}}
\newcommand{\R}{\mathbb{R}}
\newcommand{\Q}{\mathbb{Q}}

\def\ds1{\mathds{1}}
\renewcommand{\epsilon}{\varepsilon}

\renewcommand{\tilde}{\widetilde}

\newlength{\minipagewidth}
\newcommand{\beq}{\begin{equation}}
\newcommand{\eeq}{\end{equation}}

\newcommand{\beqa}{\begin{eqnarray}}
\newcommand{\eeqa}{\end{eqnarray}}

\newcommand{\beqan}{\begin{eqnarray*}}
\newcommand{\eeqan}{\end{eqnarray*}}

\def\ba#1\ea{\begin{align*}#1\end{align*}} 
\def\banum#1\eanum{\begin{align}#1\end{align}} 